\documentclass[11pt,twoside]{article}
\usepackage{amsmath,amssymb,amsthm,amsfonts,amstext,amsbsy,amscd}

\usepackage{a4wide}
\usepackage{comment}
\usepackage{graphicx}
\usepackage{mathrsfs}
\numberwithin{equation}{section}

\newtheorem{problem}{Problem}

\swapnumbers

\newtheorem{satz}{Satz}[section]

\newtheorem{proposition}[satz]{Proposition}

\DeclareMathOperator{\E}{{\mathbb E}}

\DeclareMathOperator{\R}{{\mathbb R}}
\DeclareMathOperator{\C}{{\mathbb C}}

\DeclareMathOperator{\supp}{supp}

\DeclareMathOperator{\Var}{Var} \DeclareMathOperator{\Cov}{Cov}
 \DeclareMathOperator{\dist}{dist}

\providecommand{\eps}{\varepsilon}
\renewcommand{\phi}{\varphi}
\renewcommand{\theta}{\vartheta}
\renewcommand{\subset}{\subseteq}

\providecommand{\abs}[1]{\lvert #1 \rvert}
\providecommand{\norm}[1]{\lVert #1 \rVert}

\providecommand{\babs}[1]{{\Bigl\lvert #1 \Bigr\rvert}}

\renewcommand{\Re}{\operatorname{Re}}
\renewcommand{\Im}{\operatorname{Im}}
\renewcommand{\le}{\leqslant}
\renewcommand{\ge}{\geqslant}

\begin{document}
\title{Testing the characteristics of a L\'evy process}
\author{\parbox[t]{6.5cm}{\centering Markus
 Rei\ss\\[2mm]
 \normalsize{\it
Institute of Mathematics\\
Humboldt-Universit\"at zu Berlin}\\
\mbox{} mreiss@mathematik.hu-berlin.de}}

%\date{\today}

\maketitle

\begin{abstract}
For $n$ equidistant observations of a L\'evy process at time distance $\Delta_n$ we consider the problem of testing hypotheses on the volatility, the jump measure and its Blumenthal-Getoor index in a non- or semiparametric manner. Asymptotically as $n\to\infty$ we allow for both, the high-frequency regime $\Delta_n=\frac1n$ and the low-frequency regime $\Delta_n=1$ as well as intermediate cases. The approach via empirical characteristic function unifies existing theory and sheds new light on diverse results. Particular emphasis is given to asymptotic separation rates which reveal the complexity of these basic, but surprisingly non-standard inference questions.
\end{abstract}

%{\small \noindent {\it Key words and Phrases:} .\\

\newpage

\section{Introduction}

Jump processes have become a key model in stochastics over the recent years. On one hand this is due to an increasing need for modelling stochastic processes with jumps in areas ranging from physics and biology to finance and economics. On the other hand, there is a more and more profound understanding of the theory and properties of jump processes. While a general framework is certainly provided by semimartingale theory, L\'evy processes remain the basic building blocks. Many general results are derived by showing a good approximation by a L\'evy process, a semimartingale with constant and independent characteristics, on short time intervals. The statistics for jump processes and especially L\'evy processes has consequently attracted increasing attention over the last decade and is still a very active field of research. It has fascinating links to different fields of probability, statistics and analysis and is a source of important, suggestive and highly non-trivial problems.

Here we concentrate on hypothesis testing based on observations of a L\'evy process $(X_t,t\ge 0)$. We assume throughout $\E[X_t^2]<\infty$ such that the L\'evy-Khinchine formula in terms of the characteristic triplet ${\cal T}=(\sigma^2,b,\nu)$ can be described by
\[ \E[e^{iuX_\Delta}]=\exp(\Delta\psi(u)),\quad \psi(u)=ibu-\tfrac{\sigma^2}2 u^2+\int(e^{iux}-1-iux)\nu(dx).
\]
We are interested in hypotheses on the volatility $\sigma^2\ge 0$, the drift $b\in\R$ and the L\'evy or jump measure $\nu$. In our observation model we dispose of equidistant observations $X_0,X_\Delta,\ldots,X_{n\Delta}$ at time distance $\Delta>0$. We shall adopt an asymptotic point of view and let $n\to\infty$ such that all quantities may be indexed by $n$. The distance $\Delta_n$ may remain fixed which is up to a simple normalisation the so-called low-frequency setting $\Delta_n=1$ or the observation time $n\Delta_n$ may remain fixed which is up to a factor the high-frequency setting $\Delta_n=\frac1n$ (also called in-fill asymptotics in econometrics) or $\Delta_n$ may have an asymptotic behaviour inbetween, whence we always assume $\Delta_n\in [\frac1n,1]$. The different asymptotics of $\Delta_n$ lead to very different statistical procedures and one purpose of the present work is to provide a unifying framework to reconcile these sometimes quite disjoint strands of literature.

In the pure L\'evy case, where the characteristics are well understood in the characteristic exponent $\psi$, it is natural to base the inference on the empirical characteristic function
of the i.i.d. increments $\Delta^n_kX:=X_{k\Delta_n}-X_{(k-1)\Delta_n}$ (double usage of the symbol $\Delta$ is unfortunately standard):
\[ \hat\phi_n(u)=\frac1n \sum_{k=1}^ne^{iu\Delta_k^nX}.\]
We shall deal with complex-valued random variables $Z_i$ and define $\Cov_{\C}(Z_1,Z_2)=\E[Z_1\bar Z_2]-\E[Z_1]\overline{\E[Z_2]}$, $\Var_{\C}(Z_1)=\E[\abs{Z_1-\E[Z_1]}^2]$. With $\phi_n(u)=\E[e^{iu\Delta_k^nX}]$ standard calculations  yield
$\E[\hat\phi_n(u)]=\phi_n(u)$, $\Cov_{\C}(\hat\phi_n(u),\hat\phi_n(v))=\tfrac1n(\phi_n(u-v)-\phi_n(u)\phi_n(-v))$,
$\Var_{\C}(\hat\phi_n(u))=\tfrac1n(1-\abs{\phi_n(u)}^2)$.

For the high-frequency setting inference is often based on truncating the increments $\Delta_k^nX$ such that the diffusive part is separated from larger jumps. As we shall see, there are many parallels of these two approaches and the characteristic function ansatz provides also new insights in the case where we might approximate well a continuous-time observation statistics. On the other hand, as we shall see, a direct method of moments may be well suited even for the low-frequency setting.

The test problems we consider are grouped according to the object of interest: the volatility in Section 2, the degree of jump activity in Section 3 and the global jump measure in Section 4. Inference on the drift $b$ is straightforward since in our definition $\E[\Delta_k^nX]=\Delta_n b$ and a natural and in many senses efficient estimator is given by $\hat b_n=(n\Delta_n)^{-1}X_{n\Delta_n}$. When testing for instance the volatility $\sigma^2$ we want to be as agnostic as possible with respect to the jump measure and the drift which, of course, confound our observations. In fact, we are facing nonparametric testing problems in the sense of Ingster and Suslina \cite{IngsterSus2003}: we consider on an infinite-dimensional class $\Theta$ a finite or infinite dimensional hypothesis set $\Theta_0$ and  set $\Theta_1(r_n)=\{\theta\in\Theta\,|\,\dist(\theta,\Theta_0)\ge r_n\}$ in some metric with $r_n>0$ and aim at testing $H_0:\theta\in\Theta_0$ against the local nonparametric alternatives $H_1(r_n):\theta\in\Theta_1(r_n)$. The aim is to provide a test of uniform level $\alpha\in(0,1)$, at least asymptotically, which has asymptotically power 1 uniformly over the alternative $H_1(r_n)$ for a separation rate $r_n\to 0$ as fast as possible. Since testing and estimation are strongly related, the best separation rate is often (mainly for testing functionals) equal to the best estimation rate or can be translated to it. We shall focus here mainly on these rates which quantify the complexity of the statistical problem and only touch upon asymptotically exact critical values. Deriving the latter becomes much more technical, e.g. requiring uniform central limit theorems under $H_0$, or degenerate since often an imposed bias condition on the nuisance part dominates the test statistics asymptotically. Adaptive nonparametric testing is of high importance in practice, but beyond the scope of the current work.

To the best of our knowledge a general approach for testing the quantities in the characteristic triplet of a L\'evy process has not yet been pursued. On the other hand, very profound and interesting estimation, testing and confidence statements have been derived in the literature for a variety of situations. We shall discuss the corresponding literature and their relationships in detail when presenting our results. These discussions also include open problems that seem to be of fundamental importance and show the frontiers of our knowledge on inference for jump processes, while encouraging hopefully further research in this area.

Before treating specific testing problems let us give a very rough guide to what kind of inference is possible via the characteristic function approach. We use standard notation for asymptotic quantities like $a_n=O(b_n)$ or $a_n\lesssim b_n$ if $a_n\le Cb_n$ for some constant $C>0$ independent of $n$ or $X_n=O_P(b_n)$ if $(X_n/b_n)_{n\ge 1}$ is stochastically bounded or tight. Only for heuristics we use  $a_n\gg b_n$ meaning $a_n$ is much larger than $b_n$, mostly $a_n/b_n\to\infty$.

Inference on the characteristic triplet ${\cal T}=(\sigma^2,b,\nu)$ will be based on the empirical characteristic exponent $\hat\psi_n(u):=\Delta_n^{-1}\log(\hat\phi_n(u))$, where the distinguished logarithm is taken, i.e. that branch of the complex logarithm which renders $\hat\psi_n$ continuous with $\hat\psi_n(0)=0$. A linearisation yields $\hat\psi_n(u)-\psi(u)\approx \frac{\hat\phi_n(u)-\phi_n(u)}{\Delta_n\phi_n(u)}$.
The stochastic error or noise level in $\hat\psi_n(u)$ is thus of the order
\[ \frac{\Var_{\C}^{1/2}(\hat\phi_n(u))}{\Delta_n\phi_{0,n}(u)}
=\frac{\sqrt{\abs{\phi_n(u)}^{-2}-1}}{\sqrt{n}\Delta_n}.
\]
We can reliably distinguish between two characteristics $(b_j,\sigma_j^2,\nu_j)$, $j=0,1$, at frequency $u$ if the 'signal' $(\psi_{1}-\psi_{0})(u)$ is significantly larger than the noise level in the observations generated by the characteristics with $j=0$, say, i.e.
\begin{equation}\label{EqSNR1} \abs{(\psi_1-\psi_0)(u)}\gg %\frac{\sqrt{1-\abs{\phi_{1,n}(u)}^2}}{\sqrt{n}\Delta_n\abs{\phi_{1,n}(u)}}=
\sqrt{\frac{\abs{\phi_{1,n}(u)}^{-2}-1}
{n\Delta_n^2}}.
\end{equation}
We have for any characteristics $\abs{\phi_n(u)}\gtrsim e^{-\Delta_n c^2u^2}$ for some $c>0$. Because of $\abs{\phi_{0,n}(u)}^{-2}=\exp(2\Delta_n\abs{\Re(\psi_{0}(u))})\approx 1+2\Delta_n\abs{\Re(\psi_{0}(u))}$ for $\abs{\Re(\psi_{0}(u))}\Delta_n\to0$
the requirement \eqref{EqSNR1} thus imposes the restriction
%\begin{equation}\label{EqSNR1} \abs{(\psi_1-\psi_0)(u)}\gg
%\frac{\sqrt{\Delta_n^{-1}\wedge\abs{\Re(\psi_0(u))}}}{\sqrt{n\Delta_n}\abs{\phi_{0,n}(u)}}\thicksim
%\begin{cases}
%\frac{\sqrt{\abs{\Re(\psi_0(u))}}}{\sqrt{n\Delta_n}},&\text{ for } \abs{\Re(\psi_0(u))}<\Delta_n^{-1}\\
%\frac{\abs{\phi_{0,n}(u)}^{-1}}{\sqrt{n\Delta_n^2}},&\text{ for } \abs{\Re(\psi_0(u))}\ge\Delta_n^{-1}\end{cases}
%\end{equation}
% In this general case we see that \eqref{EqSNR1} can only be satisfied whenever
\begin{equation}\label{EqSNR2}
%u^2\gg \frac{\sqrt{\Delta_n^{-1}\wedge u^2}}{\sqrt{n\Delta_n}\exp(-\Delta_n c_n^2u^2)}
%\text{, i.e. }
\frac1{\sqrt{n\Delta_n}}\ll \abs{u} < (\sqrt{2}c_n)^{-1}\sqrt{\Delta_n^{-1}\log n}.
\end{equation}
Let us suppose $c_n=c$ to be constant and then consider the two main frequency asymptotics:
\begin{enumerate}
\item (HF-estimation) $\Delta_n=1/n$: the  requirement is $1\ll \abs{u}<(\sqrt{2}c)^{-1}\sqrt{n\log n}$.

\item (LF-estimation) $\Delta_n=1$: the  requirement is $n^{-1/2}\ll \abs{u}< (\sqrt{2}c)^{-1}\sqrt{\log n}$.

%\item (intermediate) $\Delta_n=n^{-\delta_n}$ for $\Delta_n\in(0,1)$:
%the maximal frequency domain we can treat is $n^{(\delta_n-1)/2}\ll \abs{u}< (\sqrt{2}c)^{-1}n^{\delta_n/2}\sqrt{\log n}$.
\end{enumerate}

These heuristics fit well with the theory for these cases. In the high-frequency case we cannot identify, even for continuous-time observations, the jump measure away from zero (there are only finitely many large jumps occurring), but the volatility and the index of jump activity around zero are identifiable and they determine the behaviour of $\phi_n(u)$ for $\abs{u}\to\infty$. The upper bound for $\abs{u}$ shows that the stochastic error prevents us from using too large frequencies and indicates roughly a $\sqrt{n}$-rate that is at best attainable. In contrast, the low-frequency regime allows asymptotically to identify the entire triplet, but we might face logarithmic rates in $n$. If no Brownian motion is present, i.e. $\sigma=0$, the characteristic function decays less rapidly and faster convergence rates are possible.

\section{Tests on volatility}

\subsection{General idea}

A first important problem is to test the hypothesis $H_0:\sigma^2=\sigma_0^2$ for some fixed volatility $\sigma_0\ge 0$, while the drift $b$ and the jump measure $\nu$ remain largely unspecified. When passing to equivalent measures on path space, the volatility remains invariant and a hypothesis on the volatility in finance is often obtained from data in a risk-neutral setting like option price data.

The important observation now is that $-\frac{\sigma^2}{2}u^2$ is the dominant term in the characteristic exponent $\psi(u)$ for large frequencies $\abs{u}$. If we apply the general signal-to-noise ratio idea \eqref{EqSNR1} to distinguish with some alternative volatility $\sigma_1$, the frequency $U_n>0$ we choose should satisfy
\[ \abs{\sigma_1^2-\sigma_0^2}U_n^2\gg \frac{U_n}{\sqrt{n\Delta_n}}\vee \frac{e^{\Delta_n\sigma_0^2U_n^2/2}}{\sqrt{n}\Delta_n}.
\]
If there were no nuisance induced by the jump measure, we would just optimise over $U_n$ and take $U_n=\sqrt{2\Delta_n^{-1}\sigma_0^{-2}}$, which yields the parametric $n^{-1/2}$-separation rate for $\abs{\sigma_1^2-\sigma_0^2}$.

Genuinely, however, we have to take into account the impact of drift and jumps. Since the volatility only appears in $\Re(\psi(u))$, the drift does not interfere when the test is based on the real part $\Re(\hat\psi_n(u))=(\hat\psi_n(u)+\hat\psi_n(-u))/2$, but the jump measure does, of course. In the spirit of Jacod and Rei{\ss} \cite{JJReiss2012} we classify the jump activity of $\nu$ by a jump index  $\beta\in[0,2]$ of Blumenthal-Getoor type of intensity $R>0$ and consider
\[\nu_0,\nu_1\in BG(\beta,R)=\Big\{\nu\,\Big|\,\int (2^{1-\beta}\abs{x}^\beta\vee x^2)\,\nu(dx)\le R\Big\}.
\]
In terms of the characteristic exponent this assumption implies
$\Re(\psi_j(u)+\tfrac{\sigma_j^2}{2}u^2)\in [-R\abs{u}^\beta,0]$, $j=0,1$, because
\begin{align*}
&\Re\Big(\int (e^{iux}-1-iux)\,\nu_j(dx)\Big)=\int (\cos(ux)-1)\,\nu_j(dx)\\
&\quad\ge -\int \Big(\frac{u^2x^2}{2}\wedge 2\Big)\,\nu_j(dx)\ge -2\int \abs{ux/2}^\beta\nu_j(dx).
\end{align*}
For the separation of the volatilities this additional bias induced by the jump activity thus requires
\[ \abs{\sigma_1^2-\sigma_0^2}U_n^2\gg RU_n^\beta.\]
For this part, the larger $U_n$ the better and considering bias and stochastic part together, we should optimize $U_n$ such that
\[ \max\Big(\frac{U_n^{-2}\exp(\Delta_n(\sigma_0^2U_n^2/2+RU_n^\beta))}{\sqrt{n}\Delta_n},
RU_n^{\beta-2}\Big) \to\text{min!}
\]
Asymptotically for  large $U_n$, $\sigma_0>0$ and $n\Delta_n^{2-\beta}\to\infty$, this gives $U_n=\sigma_0^{-1}\Delta_n^{-1/2}\sqrt{\log(n\Delta_n^{2-\beta})}$ and the separation condition
\[ \abs{\sigma_1^2-\sigma_0^2}\gg \Big(\frac{\Delta_n\sigma_0^2}{\log(n\Delta_n^{2-\beta})}\Big)^{(2-\beta)/2}.
\]
For $n\Delta_n^{2-\beta}\lesssim 1$ the bias bound is asymptotically negligible and we choose
$U_n=\sqrt{2\Delta_n^{-1}\sigma_0^{-2}}$ as in the parametric case. The critical line $n\Delta_n^{2-\beta}=1$ reduces in the high-frequency framework $\Delta_n=\frac1n$ to the case $\beta=1$, where it separates the standard parametric-type bounds for $\beta\le 1$ with the nonparametric bounds for $\beta>1$ \cite{JJReiss2012}. For the low-frequency setting or more generally $\Delta_n\gtrsim n^{-1/2}$ the parametric type never appears.

In the case $\sigma_0=0$, the characteristic exponent under $H_0$ changes and always yields the choice $U_n=(R\Delta_n)^{-1/\beta}\log(\sqrt{n})$ which implies the faster separation rate $\sigma_1^2\gg R^{2\beta}\Delta_n^{(2-\beta)/\beta}(\log n)^{\beta-2}$.
For any $\Delta_n\to 0$ with polynomial rate in $n$ there is a $\beta>0$ such that the separation rate is of super-efficient order $o(n^{-1/2})$, which seems counterintuitive. Yet, note that in the parametric case the testing problem faces the degenerate law $N(0,0)$ under $H_0$ and thus the rate is dictated by the bias part.

\subsection{Test construction and properties}

Let us specify the first nonparametric testing problem more rigorously as
\[ H_0: {\cal T}\in\Theta_0 \text{ versus }  H_1(r_n): {\cal T}\in\Theta_1(r_n)\]
with
\begin{equation}\label{EqThetaVol}
\begin{split}
\Theta_0&=\{(\sigma_0^2,b,\nu)\,|\, b\in\R, \nu\in BG(\beta,R)\},\\ \Theta_1(r_n)&=\{(\sigma^2,b,\nu)\,|\,\abs{\sigma^2-\sigma_0^2}\ge r_n\sigma_0^2, b\in\R, \nu\in BG(\beta,R)\},
\end{split}
\end{equation}
where $\sigma_0^2>0$ and $r_n\downarrow 0$ is a separation rate that tends as fast as possible to zero while the two hypotheses can still be distinguished asymptotically.

The above discussion was based on the characteristic exponent. For a concrete construction of the test and limit theorems the characteristic function itself is more convenient. The variance result for $\hat\phi_n(u)$ implies by the inverse triangle inequality
\[ \E[(\abs{\hat\phi_n(u)}-\abs{\phi_n(u)})^2]\le \E[\abs{\hat\phi_n(u)-\phi_n(u)}^2]\le \frac1n.
\]
For the characteristic functions $\phi_{0,n}$ in our hypothesis $H_0$ we have
\[ e^{-\Delta_n(\frac{\sigma_0^2}{2}U_n^2+RU_n^\beta)}\le\abs{\phi_{0,n}(U_n)}\le e^{-\Delta_n\frac{\sigma_0^2}{2}U_n^2}.
\]
This implies by Chebyshev inequality for any  $\kappa>1$ the conservative, but uniform convergence bound
\begin{align*}
&\limsup_{n\to\infty}\sup_{{\cal T}\in\Theta_0}P_{\cal T}(\abs{\hat\phi_n(U_n)}\notin {\cal A}_{n,\kappa})\le \kappa^{-2}
\text{with }{\cal A}_{n,\kappa}=\Big[e^{-\Delta_n(\frac{\sigma_0^2}{2}U_n^2+RU_n^\beta)}-\frac{\kappa}{\sqrt{n}},
e^{-\Delta_n\frac{\sigma_0^2}{2}U_n^2} +\frac{\kappa}{\sqrt{n}}\Big].
\end{align*}
According to the preceding discussion, fix some small $\delta>0$ and put $\eps=0$ if $\Delta_n\le n^{-\delta}$ and $\eps=3\delta$ otherwise and
then choose the frequency
\begin{equation}\label{EqUVol}
U_n=\sqrt{\frac{(1-\eps)\log(n\Delta_n^{2-\beta})\vee 2}{\sigma_0^2\Delta_n}},
\end{equation}
where the additional factor $1-\eps$ will ensure uniformity over the bias part.
We arrive at the acceptance interval
\begin{align}\label{EqAnVol}
{\cal A}_{n,\kappa}&=\Big[A_n^{1+a_n}-\frac{\kappa}{\sqrt{n}}, A_n+\frac{\kappa}{\sqrt{n}}\Big]\text{ with } A_n=(n\Delta_n^{2-\beta})^{-(1-\eps)/2}\wedge e^{-1},\,a_n=2R\sigma_0^{-2}U_n^{\beta-2}.
\end{align}
%with
%\[a_n:=2R\sigma_0^{-2}U_n^{\beta-2}=2R\sigma_0^{-\beta}(\log(C^{-1}T\Delta_n^{1-\beta})\vee 2)/\Delta_n)^{\beta/2-1}.
%\]

\begin{proposition}\label{PropVol}
For the testing problem specified in \eqref{EqThetaVol} with $\sigma_0,\beta,R>0$ and separation rate
\[r_n=2a_n+\frac{\rho_n}{\sqrt{n}}
=4R\sigma_0^{-2}\Big(\frac{\sigma_0^2\Delta_n}{(1-\eps)\log(n\Delta_n^{2-\beta})\vee 2}\Big)^{(2-\beta)/2} +\frac{\rho_n}{\sqrt{n}}\text{ for some }\rho_n\to\infty
\]
the test that accepts if $\abs{\hat\phi_n(U_n)}\in{\cal A}_{n,\kappa}$ with frequency $U_n$ from \eqref{EqUVol} and acceptance interval ${\cal A}_{n,\kappa}$ from \eqref{EqAnVol} has asymptotic level at most $\kappa^{-2}$ uniformly over the hypothesis $H_0$ and asymptotic power 1 over the alternative $H_1(r_n)$.
\end{proposition}

\begin{proof}
It remains to consider the asymptotics under the alternative. In the case $\eps>0$ we have $A_n\ge n^{-(1-\eps)/2}$, in the case $\Delta_n\le n^{-\delta}$ we have $A_n\ge n^{-(1-(2-\beta)\delta)/2}$, whence from $a_n\to 0$ we infer $\sqrt{n}A_n^{1+a_n}\to\infty$ and the left endpoint of ${\cal A}_{n,\kappa}$ tends to infinity.  Consider now a triplet from the alternative with $\sigma_1^2\ge \sigma_0^2(1+r_n)$ such that its characteristic function satisfies
\[\abs{\phi_{1,n}(U_n)}\le e^{-\Delta_n\frac{\sigma_1^2}{2}U_n^2}=A_n^{1+r_n}.
\]
In the case $\Delta_n>n^{-\delta}$ we have $a_n\gtrsim \Delta_n>n^{-\delta}$ and thus $\sqrt{n}A_n^{1+r_n} r_n\gtrsim n^{\eps/2-\delta-o(1)}\to\infty$ since $\eps/2-\delta>0$. In the case $\Delta_n\le n^{-\delta}$ ($\eps=0$) we find
\begin{align*}
&A_n\sqrt{n}(r_n-a_n)\log(A_n^{-1})\\
&\qquad\gtrsim (n\Delta_n^{2-\beta}\vee 1)^{-1/2}((n\Delta_n^{-(\beta-2)})^{1/2}(\log(n\Delta_n^{2-\beta})\vee 1)^{(\beta-2)/2}+\rho_n)(\log(n\Delta_n^{2-\beta})\vee 1)\\
&\qquad\thicksim (\log(n\Delta_n^{2-\beta}))^{\beta/2}+(n\Delta_n^{2-\beta})^{-1/2} \rho_n\to\infty.
\end{align*}
Furthermore, note $A_n^{r_n}\to 1$ for $\log r_n\lesssim -\log n$ (satisfied in case $\eps=0$). Putting these asymptotics together,
the distance to the acceptance interval, weighted by $\sqrt{n}$, tends to infinity using $x^p-1\ge p\log(x)$ for $x,p>0$:
\begin{align*}
\sqrt{n}(A_n^{1+a_n}-A_n^{1+r_n})&\ge \sqrt{n}A_n^{1+r_n}(r_n-a_n)\log(A_n^{-1})
\to \infty.
\end{align*}
A completely symmetric analysis for a triplet from the alternative with $\sigma_1^2\le \sigma_0^2(1-r_n)$ yields the same order for the distance to the acceptance interval. Consequently, we have
\[ \lim_{n\to\infty}\inf_{{\cal T}\in\Theta_1(r_n)}P_{\cal T}(\abs{\hat\phi_n(U_n)}\notin {\cal A}_{n,\kappa})=1
\]
and the power of the test on the alternative $H_1$ uniformly tends to one.
\end{proof}

For $\sigma_0=0$ we change the hypotheses by setting
\begin{equation}\label{EqThetaVol0} \Theta_0=\{(0,b,\nu)\,|\, b\in\R, \nu\in BG(\beta,R)\},\quad \Theta_1(r_n)=\{(\sigma^2,b,\nu)\,|\,\sigma^2\ge r_n, b\in\R, \nu\in BG(\beta,R)\}
\end{equation}
and we consider
\begin{equation}\label{EqUAnVol0}
U_n=(R\Delta_n)^{-1/\beta}(\log(\sqrt{n}/(2\kappa)))^{1/\beta},\quad {\cal A}_{n,\kappa}=[\kappa n^{-1/2}, \infty).
\end{equation}

\begin{proposition}
For the testing problem specified in \eqref{EqThetaVol0} with $\beta>0$ and any separation rate $r_n\to 0$ satisfying
\[ r_n-(2R)^{2/\beta}\Big(\frac{\Delta_n}{\log n}\Big)^{(2-\beta)/\beta}\to \infty \]
the test that accepts if $\abs{\hat\phi_n(U_n)}\in{\cal A}_{n,\kappa}$ with $U_n$ and ${\cal A}_{n,\kappa}$ from \eqref{EqUAnVol0} has asymptotic level $\kappa^{-2}$ uniformly over the hypothesis $H_0$ and asymptotic power 1 over the alternative $H_1(r_n)$.
\end{proposition}

\begin{proof}
Under the hypothesis $H_0$ we just use the bias bound
\[ \abs{\phi_{0,n}(U_n)}\ge e^{-\Delta_nRU_n^\beta}=2\kappa n^{-1/2}
\]
and the same deviation bound (by Chebyshev's inequality) as before to attain the uniform level. Under the alternative we have
\[ \sqrt{n}\abs{\phi_{1,n}(U_n)}\le e^{(\log n-\Delta_nr_nU_n^2)/2}\to 0\]
and we conclude again by applying the uniform deviation bound.
\end{proof}

\subsection{Discussion and extensions}

Due to the semi-parametric feature of the testing problem the asymptotic level of the test is given by the limit of the worst case error probability of the first kind over $H_0$ (in this order!). If for each $n$ only a single hypothesis is tested against a single alternative, then a parametric $n^{-1/2}$-separation rate can be achieved. This is discussed in detail by Ait-Sahalia and Jacod \cite{ASJJ2007} where even the exact LAN-property is derived and instead of Brownian motion also any stable process can be taken that dominates the remaining jump part. For their rather involved estimation procedure the authors then also prove a nonparametric $n^{-1/2}\wedge \Delta_n^{(2-\beta)/2}$-rate, which is up to the loss of the logarithmic factor exactly our separation rate $r_n$. As discussed in Jacod and Rei{\ss} \cite{JJReiss2012} for the high-frequency case, this logarithmic gain is possible also for estimation, the estimator simply being
\[ \hat\sigma_n^2:=-2\Re(\log(\hat\phi_n(U_n)))\text{ with our choice of }U_n.\]
The rate-optimality of this estimator has been established and the lower bound proof directly delivers the lower bound for our separation rate in Proposition \ref{PropVol}. Testing the hypothesis $H_0:\sigma^2=0$ has been studied by Ait-Sahalia and Jacod \cite{ASJJ2010} in a much wider semi-martingale setting, but reduced to the L\'evy case their results are  less precise than ours. It would be interesting to study the super-efficiency of rates $o(n^{-1/2})$ for less restrictive models than pure L\'evy processes, e.g. allowing for time-varying, but deterministic characteristics.

Another interesting observation is that the Blumenthal-Getoor-type index $\beta$ is only used to guarantee the order $\abs{U_n}^\beta$ for the bias term in the characteristic exponent. In the case of finite intensity and a jump density (with respect to Lebesgue measure) $\nu$, the Fourier transform satisfies $\abs{{\cal F}\nu(U_n)}=O(U_n^{-s})$ whenever $\nu\in C^s(\R)$ with integrable derivatives of all orders. In that case the test statistics should be based upon $\int_0^1\Re(\hat\psi_n(uU_n))w(u)du$ with a weight function $w$ satisfying $\int_0^1 u^2w(u)du=2$ (to recover $\sigma^2$), $\int_0^1 w(u)du=0$ (to filter out $\lambda=\nu(\R)$) and $\sup_{u\in[0,1]}u^{-s}\abs{w(u)}<\infty$ (to ensure $\int \abs{uU_n}^{-s}w(u)du=O(U_n^{-s})$). This approach ensures that the bias is of order $O(U_n^{-s})$ while the stochastic error is not increased since only frequencies up to $U_n$ enter. Testing and estimation of $\sigma$ (in the case $\sigma>0$) is therefore possible with rate $(\frac{\Delta_n}{\log(n\Delta_n)})^{(2+s)/2}\vee \frac{1}{\sqrt{n}}$. This is particularly interesting in the low-frequency regime $\Delta_n=1$ with rate $(\log n)^{-(s+2)/2}$. In the related setting of estimation from financial option data the weight function approach has been successfully analysed and applied by Belomestny and Rei{\ss} \cite{BelReiss2006}. S{\"o}hl \cite{Soehl2012} addresses testing and confidence regions for this situation in detail and S{\"o}hl and Trabs \cite{SoTra2012} apply this to financial data. Generalisations to other processes defined by their Fourier transform are possible, e.g. see Belomestny \cite{Bel2011} for estimation in affine processes.

An accurate choice of the critical value $\kappa$ to attain a certain asymptotic level should be based on a uniform central limit theorem. For empirical characteristic functions $\hat\phi_n$ and corresponding characteristic functions $\phi_n$  a central limit theorem for triangular schemes yields
\[ \sqrt{n}\Sigma_n^{-1/2}(\Re(\hat\phi_n(U_n)-\phi_n(U_n)),\Im(\hat\phi_n(U_n)-\phi_n(U_n)))\to N_{\R^2}(0,I)
\]
with $\Sigma_n\in\R^{2\times 2}$ equal to
\[\frac12\begin{pmatrix}1+\Re(\phi_n(2U_n))-2\Re(\phi_n(U_n))^2 & \Im(\phi_n(2U_n))-2\Re(\phi_n(U_n))\Im(\phi_n(U_n))\\
\Im(\phi_n(2U_n))-2\Re(\phi_n(U_n))\Im(\phi_n(U_n)) & 1-\Re(\phi_n(2U_n))-2\Im(\phi_n(U_n))^2\end{pmatrix}.
\]
The form of $\Sigma_n$ can be easily checked using trigonometric identities, e.g. $\E[\cos(uX_k)^2]=\frac12(1+\E[\cos(2uX_k)])$ for the variance of the real part.
%The $\Delta$-method with the mapping $r(x,y)=\sqrt{x^2+y^2}$, $\nabla g(x,y)=(x,y)/g(x,y)$ therefore yields
%\[ \sqrt{n}(\tfrac12(1-2\abs{\phi_n(U_n)}^2
%+\Re(\overline{\phi_n(2U_n)}\phi_n(U_n)^2/\abs{\phi_n(U_n)})))^{-1/2}
%(\abs{\hat\phi_n(U_n)}-\abs{\phi_n(U_n)})\to N_{\R}(0,1).
%\]
In the case $n\Delta_n^{2-\beta}\to\infty$ both $\phi_n(U_n),\phi_n(2U_n)$ tend to zero whereas $n\Delta_n^{2-\beta}\to 0$ implies $\phi_n(U_n)\to e^{-1}$, $\phi_n(2U_n)\to e^{-4}$
(we omit the intermediate case). We apply the $\Delta$-method for the mapping $(x,y)\mapsto \sqrt{x^2+y^2}$ and conclude for our setting of testing under $H_0:\sigma^2=\sigma_0^2>0$
\[ \sqrt{n}(\abs{\hat\phi_n(U_n)}-\abs{\phi_n(U_n)})\to N_{\R}(0,v^2) \text{ with } v=\begin{cases} 1/2,&\text{ if }n\Delta_n^{2-\beta}\to\infty,\\
(1-e^{-2})^2/2,&\text{ if }n\Delta_n^{2-\beta}\to 0.\end{cases}
\]

To attain an asymptotic level $\alpha\in(0,1)$ we should therefore choose $\kappa=vq_{1-\alpha/2}$ with the $(1-\alpha/2)$-quantile of $N(0,1)$. Remark that based upon $\abs{\hat\phi_n(U_n)}$ the test is asymptotically distribution free for the composite hypothesis $H_0$. Let us recall that uniformity in the central limit theorem over the null hypothesis is not established here.

For general semi-martingale models our simple characteristic function approach does not allow inference on the integrated volatility $\int_0^1\sigma_t^2dt$, assuming here $\Delta=\frac1n$ for simplicity. The empirical characteristic function of the increments, however, still estimates the realized Laplace transform $\int_0^1e^{-u^2\sigma_t^2/2}dt$ for stochastic volatility $\sigma_t^2$. Tauchen and Todorov \cite{TT2012}
provide a uniform (in $u$) central limit theorem for this case under jumps of bounded variation and discuss the potential of this Laplace transform approach for applications. For integrated volatility inference a direct extension of the present method is possible if the local characteristics of the semi-martingale vary smoothly and we localize the test statistics on small blocks.

In the area of high-frequency financial statistics power variation methods are traditionally used, comparing discrete $p$-variations for continuous martingales and jump processes. It is quite striking that a slightly smaller choice of the frequency $U_n$ transforms our test statistics to a truncated quadratic variation statistics.
If $\delta_n:=U_n^2\Delta_n\to 0$, we have for any $L^2$-It\^o-semimartingale $X$ with bounded characteristics
\begin{align*}
&\hat\phi_n(U_n)-1 = \frac1n\sum_{k=1}^n e^{iU_n\Delta_k^nX}({\bf 1}_{\{\abs{\Delta_k^nX}\le\delta_n^{1/4}U_n^{-1}\}}+{\bf 1}_{\{\abs{\Delta_k^nX}>\delta_n^{1/4}U_n^{-1}\}})-1\\
&=\frac1n\Big(\sum_{k=1}^n (iU_n\Delta_k^nX-\tfrac{U_n^2}{2}(\Delta_k^nX)^2){\bf 1}_{\{\abs{\Delta_k^nX}\le\delta_n^{1/4}U_n^{-1}\}}\Big)+O(\delta_n^{1/2})
+\frac1n\#\{k:\abs{\Delta_k^nX}>\delta_n^{1/4}U_n^{-1}\}\\
&=\frac{iU_n}{n}\sum_{k=1}^n \Delta_k^nX {\bf 1}_{\{\abs{\Delta_k^nX}\le\delta_n^{1/4}U_n^{-1}\}} -\frac{U_n^2}{2n}\sum_{k=1}^n (\Delta_k^nX)^2{\bf 1}_{\{\abs{\Delta_k^nX}\le\delta_n^{1/4}U_n^{-1}\}}+O_P(\delta_n^{1/2}
+\Delta_n\delta_n^{-1/2}U_n^{2})\\
&=\frac{iU_n}{n}\sum_{k=1}^n \Delta_k^nX{\bf 1}_{\{\abs{\Delta_k^nX}\le\delta_n^{1/4}U_n^{-1}\}} -\frac{U_n^2}{2n}\sum_{k=1}^n (\Delta_k^nX)^2{\bf 1}_{\{\abs{\Delta_k^nX}\le\delta_n^{1/4}U_n^{-1}\}}+O_P(\delta_n^{1/2}).
\end{align*}
Because of $\sum_{k=1}^n(\Delta_k^nX)^2=O_P(n\Delta_n)$ this implies
\[\abs{\hat\phi_n(U_n)}^2=1-\frac{U_n^2}{n}\sum_{k=1}^n (\Delta_k^nX)^2{\bf 1}_{\{\abs{\Delta_k^nX}\le\delta_n^{1/4}U_n^{-1}\}}+\frac{U_n^2}{n^2}\Big(\sum_{k=1}^n \Delta_k^nX{\bf 1}_{\{\abs{\Delta_k^nX}\le\delta_n^{1/4}U_n^{-1}\}}\Big)^2+o_P(1).
\]
The second sum, which is not familiar for standard high-frequency variation estimation, is just an (asymptotically negligible) bias correction. Note that for inference purposes the size of the approximation error is essential and the order $O_P(\delta_n^{1/2})=o_P(1)$ will often be sub-optimal and a more subtle argument than a second order Taylor expansion should be used. Still, the derivation shows conceptually that we recover the truncated realized variance estimator when we choose frequencies $U_n=o(\Delta_n^{-1/2})$, which correspond exactly to standard truncation levels of larger order than $n^{-1/2}$ in the high-frequency setting, see e.g. Ait-Sahalia and Jacod \cite{ASJJ2007}. Let us close this section by formulating two open problems that show our lack of fundamental understanding when passing from the pure L\'evy to the general semi-martingale setting.

%\begin{problem}
%Consider a time-inhomogeneous L\'evy process with instantaneous characteristics $(\sigma_t^2,b_t,\nu_t)$, which are bounded and continuous (or more or less regular). For $\Delta_n=\frac1n$ and $\beta>1$ find the optimal minimax rates for testing or estimating $\int_0^1\sigma_t^2dt$.
%\end{problem}

\begin{problem}
Prove or disprove that in the case of semi-martingales with $\Delta_n=\frac1n$ and $\beta=1$ (bounded variation of the jump part) the integrated volatility $\int_0^1\sigma_t^2dt$ can be tested or estimated at the parametric rate $n^{-1/2}$, as is the case for L\'evy processes or for semimartingales with $\beta<1$.
\end{problem}

\begin{problem}
Assume that the observations are $Y_k=X_{k\Delta_n}+\eps_k$ with noise variables $\eps_k\sim N(0,1)$ i.i.d. independent from $X$. Find the optimal separation rates for testing based on these noisy observations.
\end{problem}

%For the first problem an estimator or test might be based on a localized version in time of the procedures for the pure L\'evy case and then yield the rate $(n\log n)^{(\beta-2)/2}$, but the smoothness assumptions required are not obvious. Conversely, a lower bound result for slower convergence without additional smoothness assumptions is certainly not standard.
The first problem seems to be completely open and quite fundamental.
The second problem is of high interest for financial statistics under microstructure noise assumptions and even for $b=0$, $\nu=0$ the optimal rate is $n^{-1/4}$ in the case $\Delta_n=\frac1n$ (Gloter and Jacod \cite{GloterJJ2001}). The combination of microstructure noise and nuisance by jumps for inference on the integrated volatility is a very active area of current research and definite answers even for the simple L\'evy case still lack.

\section{Testing indices of jump activity}

Another stream of recent research has concentrated on the estimation of the Blumenthal-Getoor index or other jump activity indices for L\'evy processes or more general semi-martingales. The rationale for infinite activity measures like the Blumenthal-Getoor index is again that these indices remain invariant when changing to equivalent measures and they can be recovered already in the asymptotics of a high-frequency setting on a bounded observation interval. The basic idea is that these indices dictate the behaviour of $\phi_n(U_n)$ for large frequencies $U_n$. In contrast to inference on the volatility, however, two sources of nuisance enter: the terms due to volatility and due to the general form of the jump measure. We concentrate on two prototypical testing problems in this context.

\subsection{Testing the presence of jumps}

The easiest inference question is to test $H_0:\nu=0$ against the presence of a jump measure. We quantify the contribution of $\nu$ in the alternative by $H_1:\int x^4\nu(dx) \in [r_n,\infty]$ with $\sigma^8+\int x^8\nu(dx)\le R^8$ for some $R>0$ and otherwise arbitrary $\sigma^2\ge 0$ and $b\in\R$. In order to be able to separate the alternative even for finite jump activity $\nu(\R)<\infty$, we suppose $n\Delta_n\to\infty$. At first sight this problem requires the characteristic function approach, but surprisingly it is solved very easily by a direct moment method.  The formulation of the alternative is tailor-made for our purposes. For general characteristics the following curtosis equation follows directly from $(\log\phi_n)^{(4)}(0)=\Delta_n\int x^4\nu(dx)$ if that expression is finite:
\begin{equation}\label{EqCurtosis} \E[(\Delta_k^nX-\E[\Delta_k^nX])^4]- 3\E[(\Delta_k^nX-\E[\Delta_k^nX])^2]^2=\Delta_n\int x^4\nu(dx)\ge 0.
\end{equation}
Consequently, the corresponding empirical moments  have a stochastic error of order $O_P(n^{-1/2}(\Delta_n\int x^8\nu(dx)+\Delta_n^4(\sigma^2+\int x^2\nu(dx))^2))$. The latter is always of order $O_P(R^4n^{-1/2}\Delta_n^{1/2})$, but note that the error level under the hypothesis for $\Delta_n<1$ is much smaller than under the alternative. The test  accepts $H_0$ if the empirical moment in \eqref{EqCurtosis} is within $[-\kappa R^4\Delta_n^2,\kappa R^4\Delta_n^2]$  for some fixed critical value $\kappa>0$.  The separation rate is $\Delta_n^{-1}n^{-1/2}\Delta_n^{1/2}$ in the sense that for $r_n(n\Delta_n)^{1/2}\to\infty$ it has asymptotic power 1 on the alternative.

To see that the separation rate $(n\Delta_n)^{-1/2}$ is minimax optimal, just consider the test between the two single hypotheses $H_0:\sigma=0,b=0,\nu=0$
(process is constant zero) and $H_1:\sigma=0,b=0,\nu=(n\Delta_n)^{-1}\delta_{(n\Delta_n)^{1/8}}$ ($(n\Delta_n)^{1/8}$ times a Poisson process of intensity $(n\Delta_n)^{-1}$). Then with probability $e^{-1}$ under the alternative the process remains equal to zero on $[0,n\Delta_n]$ and thus no test can distinguish between $H_0$ and $H_1$. The alternative jump measure obviously satisfies $\int x^4\nu(dx)=(n\Delta_n)^{-1/2}$ as well as $\int x^8\nu(dx)=1$. Any test of level less than $e^{-1}$ therefore must have power less than 1 (both uniformly in $n$).

The lower bound argument with point masses that shift to infinity also shows that the specification of the alternative has a great impact on the optimal solution. If the alternatives separate like $\nu(\R)\in[r_n,\infty]$ or $\int x^2\nu(dx)\in[r_n,\infty]$, they do not separate at all from $H_0$ since infinitely divisible distributions whose L\'evy measures satisfy $x^2\nu_n(dx)\to\delta_0(dx)$ weakly converge to a normal distribution and thus Brownian motion cannot be separated from pure jump processes with $\int (x^2\wedge 1)\nu(dx)=1$. Let us refer to Neumann and Rei{\ss} \cite{NeuReiss2009} for a similar argument in total variation distance and the implication that (honest) testing between general classes of pure jump processes and diffusion processes is not possible, even in a high-frequency setting.

\subsection{Testing an index of Blumenthal-Getoor type}

Informally we want to test the hypothesis $H_0$ that the Blumenthal-Getoor index $\beta$ of the jump part equals $\beta_0\in[0,2)$ against the alternative $H_1$ that $\abs{\beta-\beta_0}\ge r_n$ for unspecified drift $b$ and volatility $\sigma^2\ge 0$. For stable and tempered stable processes it is known that the Blumenthal-Getoor index exactly describes the growth behaviour $\abs{U_n}^\beta$ for the jump part of the characteristic exponent as the frequency $\abs{U_n}$ tends to infinity. Pragmatically, we just generalise further to all L\'evy processes such that the characteristic exponent for frequencies $U_n\to\infty$ satisfies under $H_0$
\[ \Re(\psi(U_n))=-\frac{\sigma^2}{2}U_n^2-R\abs{U_n}^{\beta_0}-\tau(U_n)\text{ with } \abs{\tau(u)}\le R'(1+\abs{u}^{\beta'}),
\]
where we suppose that the remainder $\tau(u)$ is for large frequencies of lower order, that is $\beta'<\beta$, cf. Belomestny \cite{Bel2010} for a broader discussion of this assumption. The index $\beta_0$ appears in the second-largest term and we have to filter out the quadratic term in the case $\sigma^2>0$. As mentioned earlier, filtering can be achieved by considering a weighted mean over $\psi$. To this end, we consider general finite weighting measures $w$ on the Borel sets of $[0,1]$ with $\int_0^1 u^2w(du)=0$ and $W(\beta_0):=\int_0^1 u^{\beta_0}w(du)>0$. The weighted characteristic exponent over $[0,U_n]$ then satisfies
\[ \int_0^1 \Re(\psi(U_nu))\,w(du)=-RW(\beta_0)U_n^{\beta_0}-\int_0^1\tau(U_nu)\,w(du).
\]
We have $\abs{\int_0^1\tau(U_nu)\,w(du)}\le R'(1+U_n^{\beta'})\norm{w}_{TV}$ and thus the term of interest  dominates. Our test should be agnostic with respect to the additional parameter $R$. Belomestny \cite{Bel2010} uses the natural trick to consider the logarithm in which case $\log(\int_0^1 \Re(\psi(U_nu))\,w(du))=\beta_0\log(U_n)+O(1)$ holds. For his low-frequency situation the only logarithmically  smaller bias term does not harm, but for our general sampling scheme this is highly suboptimal. A better idea is to use two weighting measures $w_1,w_2$ with the properties of $w$ above and with different values $W_1(\beta_0)\not=W_2(\beta_0)$ such that
\[ \frac{\int_0^1 \Re(\psi(U_nu))\,w_1(du)}{\int_0^1 \Re(\psi(U_nu))\,w_2(du)}=\frac{W_1(\beta_0)+O(R'R^{-1}U_n^{\beta'-\beta_0})}
{W_2(\beta_0)+O(R'R^{-1}U_n^{\beta'-\beta_0})}.
\]
In terms of $\Re(\hat\psi(u))=\Delta^{-1}\log(\abs{\hat\phi_n(u)})$ our test statistics is therefore
\[ T_n:=\frac{\int_0^1 \Re(\hat\psi(U_nu))\,w_1(du)}{\int_0^1 \Re(\hat\psi_n(U_nu))\,w_2(du)}-Q(\beta_0)
\]
with the quotient $Q(\beta):=W_1(\beta)/W_2(\beta)$. For all indices $\beta\in[0,2)\setminus\{\beta_0\}$  we should have $\abs{Q(\beta)-Q(\beta_0)}\ge c\abs{\beta-\beta_0}$ for some $c>0$ such that under the alternative the test separates well.
A possible choice, inspired by Belomestny \cite{Bel2010}, is to take point measures \[w_j=\delta_{\eta_j}-\eta_j^2\delta_1,\quad j=1,2,\]
with different $\eta_1,\eta_2\in(0,1)$. Then $Q(\beta)=(\eta_1^\beta-\eta_1^2)/(\eta_2^\beta-\eta_2^2)$ satisfies the separation requirement (its derivative is continuous and does not change sign).

Consider the test problem $H_0:{\cal T}\in\Theta_0$ versus $H_1:{\cal T}\in \Theta_{1,n}$ for
\begin{align}
\Theta_0&=\Big\{{\cal T}\in{\cal D}(\bar\sigma)\,\Big|\,\abs{\Re(\psi_{\cal T}(u))+\tfrac{\sigma^2}{2}u^2+R\abs{u}^{\beta_0}}\le KR(1+\abs{u}^{\beta_0-\rho}),\,R>0\Big\},\label{EqTheta0BGI}\\
\Theta_{1,n}&=\Big\{{\cal T}\in{\cal D}(\bar\sigma)\,\Big|\,\abs{\Re(\psi_{\cal T}(u))+\tfrac{\sigma^2}{2}u^2+R\abs{u}^\beta}\le KR(1+\abs{u}^{\beta-\rho}),\,\abs{Q(\beta)-Q(\beta_0)}\ge r_n,R> 0\Big\}\label{EqTheta1BGI}
\end{align}
within triplets satisfying a uniform maximal decay condition
\[ {\cal D}(\bar\sigma)=\{{\cal T}\,|\,\forall\,u\in\R:\,\Re(\psi_{\cal T}(u))\in[-\tfrac{\bar\sigma^2}2 (1+u^2),0]\}
\]
and with some $\beta_0\in(0,2)$, a separation rate $r_n\downarrow 0$ and $K,\bar\sigma^2>0$, $\rho\in (0,\beta_0]$.  The decay bound involving $\bar\sigma^2$ is used to bound the stochastic error uniformly. In practice and also in theory, it may be replaced by a minor overestimate of the true volatility.
The values $K$ and $\rho$ are needed to bound uniformly the bias due to the second largest jump index. Regarding $K$ note that a very small level $R$ for the leading index $\beta$ might be confounded by a high level $R'$ of the second index $\beta-\rho$.  We get the following first result where the bias bound for $T_n$ dominates the stochastic part.

\begin{proposition}
Suppose $n\Delta_n^{2+\rho-\beta_0}\to\infty$. Consider the test on \eqref{EqTheta0BGI} and \eqref{EqTheta1BGI} that rejects if
\[ \abs{T_n}>\frac{K+\eps}{W_2(\beta_0)}(U_n^{-\beta_0}+U_n^{-\rho}
(\norm{w_1}_{TV}+\norm{w_2}_{TV}Q(\beta_0)))
\]
with $U_n^2=\Delta_n^{-1}\bar\sigma^{-2}\log (n\Delta_n^{2+\rho-\beta_0})$. Then the test
has asymptotic level $0$ and for the separation rate $r_n\to 0$ such that
\[ \frac{r_n}{\Delta_n^{\rho/2}(\log n)^{-\rho/2}}\to\infty\]
the test attains on $\Theta_{1,n}$ the asymptotic power 1.
\end{proposition}

\begin{proof}
For ${\cal T}\in\Theta_0\cup\Theta_{1,n}$ with general jump index $\beta$ we obtain from the above arguments  the uniform bias bound
\begin{align}\label{EqBiasBGI1}
&\babs{\frac{\int_0^1 \Re(\psi(U_nu))\,w_1(du)}{\int_0^1 \Re(\psi(U_nu))\,w_2(du)}-Q(\beta)}
\le
\tfrac{K}{W_2(\beta)+O(U_n^{-\rho})}(U_n^{-\beta}+U_n^{-\rho}(\norm{w_1}_{TV}+\norm{w_2}_{TV}Q(\beta))).
\end{align}
The stochastic error on the good event
\begin{equation}\label{EqGn} {\cal G}_n:=\{\forall u\in[0,1]:\abs{\hat\phi_n(U_nu)}\ge \tfrac12\abs{\phi_n(U_nu)}\}
\end{equation}
can be bounded via $\hat\psi_n-\psi=\log(1+(\hat\phi_n-\phi_n)/\phi_n)$ and Taylor expansion by
\[ \babs{\hat\psi_n(U_nu)-\psi(U_nu)-\frac{\hat\phi_n(U_nu)-\phi_n(U_nu)}
{\Delta_n\phi_n(U_nu)}}\le 2\frac{\abs{\hat\phi_n(U_nu)-\phi_n(U_nu)}^2}
{\Delta_n\abs{\phi_n(U_nu)}^2}.
\]
This implies
\begin{align*}
&\E\Big[\Big(\int_0^1 \Re(\hat\psi_n(U_nu)-\psi(U_nu))\,w_j(du)\Big)^2{\bf 1}_{{\cal G}_n} \Big]\le n^{-1}\Delta_n^{-2}\norm{w_j}_{TV}^2 \max_{u\in[0,1]} \abs{\phi_n(U_nu)}^{-2}(1+o(1)).
\end{align*}
For our choice of $U_n\to\infty$
\[\max_{u\in[0,1]}\abs{\phi_n(U_nu)}^{-2}\le \exp(\Delta_n\bar\sigma^2(1+U_n^2))\lesssim n\Delta_n^{2+\rho-\beta_0}
\]
holds. We obtain
the uniform bound under $H_0$
\begin{align*}
&\babs{\frac{\int_0^1 \Re(\hat\psi_n(U_nu))\,w_1(du)} {\int_0^1 \Re(\hat\psi_n(U_nu))\,w_2(du)}- \frac{\int_0^1 \Re(\psi(U_nu))\,w_1(du)} {\int_0^1 \Re(\psi_n(U_nu))\,w_2(du)}}{\bf 1}_{{\cal G}_n}\\
&\qquad\qquad\lesssim U_n^{-\beta_0}\max_{j=1,2}\babs{\int_0^1 \Re(\hat\psi_n(U_nu)-\psi(U_nu))\,w_j(du)}\\
&\qquad\qquad=O_P(U_n^{-\beta_0}\Delta_n^{-1}n^{1/2}\Delta_n^{1+(\rho-\beta_0)/2})
=o_P(U_n^{-\rho}),
\end{align*}
where the small $o_P$ is due to $\Delta_n^{-1}=o(U_n^2)$ for $n\Delta_n^{2+\rho-\beta_0}\to\infty$. Thus, the bias bound from \eqref{EqBiasBGI1} dominates, which defines up to an inflation by $\eps$ the acceptance interval, and the level of the test tends to zero, provided $P({\cal G}_n)$ tends to one under $H_0$. Because of $\min_{u\in[0,1]}\abs{\phi_n(U_nu)}\ge n^{-1/2+\eps/4}$ for large enough $n$, which is of larger order than $\sqrt{(\log n)/n}$, we infer $P({\cal G}_n)\to 1$ under $H_0$ from the uniform deviation bound for the empirical characteristic process, proved in Theorem 4.1 of \cite{NeuReiss2009} using bracketing entropy.

The power result is directly obtained from the corresponding bias bound under the separation rate since the same stochastic error bound shows that the bias part is dominating.
\end{proof}

In the low-frequency setting $\Delta_n=1$ the rate $r_n$ is exactly that of Theorem 6.7 in Belomestny \cite{Bel2010}, which is proved there to be minimax optimal. In general, the condition $n\Delta_n^{2+\rho-\beta_0}\to\infty$ seems artificial, however. In the high-frequency regime $\Delta_n=\frac1n$ it requires $\beta_0-\rho>1$, which in particular excludes jump indices of order less than one. On the other hand, we know from Ait-Sahalia and Jacod \cite{ASJJ2012} that also $\beta_0\le 1$ can be estimated and thus tested. The fundamental reason for our suboptimal result in this situation is that we have not used the special properties of the empirical characteristic function under a dominating Brownian motion part.

In fact, instead of allowing for positive volatility exactly the same approach could also allow for a dominating $\alpha$-stable process with $\alpha>\beta_0$, when the weights satisfy the filtering property $\int_0^1u^\alpha w(du)=0$. Since the characteristic function then decays more slowly, we choose $U_n$ of order $(\Delta_n^{-1}\log(n\Delta_n^{2+\rho-2\beta_0/\alpha}))^{1/\alpha}$, assuming $n\Delta_n^{2+\rho-2\beta_0/\alpha}\to\infty$, and obtain the separation rate $(\Delta_n/\log n)^{\rho/\alpha}$. For $\Delta_n=\frac1n$ we thus impose $\beta_0>(1+\rho)\alpha/2$, which requires $\beta_0>\alpha/2$. The latter is now a very natural condition because it is well known (e.g., Sato \cite{Sato1999}) that for $\beta_0<\alpha/2$  the laws of a (tempered) $\alpha$-stable and of the sum of a (tempered) $\alpha$-stable and a (tempered) $\beta_0$-stable process on the path space $D([0,1])$ are equivalent, which statistically means that these two process models cannot be completely separated by a test even based on continuous-time observations on $[0,1]$.

The distinction between a Brownian motion and an $\alpha$-stable process with $\alpha$ close to 2 seems statistically not significant after all we have seen so far, but the opposite is the case. The difference does not concern the 'signal' in the characteristic function, but the factorisation of the covariance function of the empirical characteristic function. We have
\[ \frac{\Cov_{\C}(\hat\phi_n(u),\hat\phi_n(v))}{\phi_n(u)\phi_n(-v)}
=\frac1n
\Big(\exp\Big(\Delta_n(\sigma^2uv+R(\abs{u}^\beta+\abs{v}^\beta-\abs{u-v}^\beta)
+\tau(u)+\tau(v)-\tau(u-v))\Big)-1\Big).
\]
Expanding the exponential this evaluates for $\abs{u},\abs{v}\to\infty$ with $\abs{u}+\abs{v}=o(\Delta_n^{-1/2})$ to
\begin{equation}\label{EqCovExpans}
\frac1n \Big(O(\Delta_n(\abs{u}^\beta+\abs{v}^\beta))+\sum_{k=1}^{2m-1}\frac{(\Delta_n\sigma^2uv)^k}{k!}+O(\Delta_n^{2m}(uv)^{2m})\Big).
\end{equation}
Let us profit from the product structure in $u$ and $v$ of the summands.
For $U_n=o(\Delta_n^{-1/2})$ and finite weight measures $w_j$, $j=1,2$,  on the Borel sets of $[-1,1]$ satisfying
\begin{equation}\label{Eqwj2}
w_j(-B)=w_j(B), \quad\int_0^1 u^{2p}w_j(du)=0\text{ for }p=1,\ldots,m-1
\end{equation}
we obtain on the good event ${\cal G}_n$ from \eqref{EqGn} the estimate
\begin{align*}
&\E\Big[\babs{\int_{0}^{1}\Re(\hat\psi_n(U_nu)-\psi(U_nu))w(u)du}^2{\bf 1}_{{\cal G}_n}\Big]\\
&\quad \le
\E\Big[\babs{\int_{-1}^{1}\Delta_n^{-1}\Big(\frac{\hat\phi_n(U_nu)-\phi_n(U_nu)}{\phi_n(U_nu)}
+O\Big(\Big(\frac{\hat\phi_n(U_nu)-\phi_n(U_nu)}{\phi_n(U_nu)}\Big)^2\Big)\Big)w_j(du)}^2\Big]\\
&\quad =\Delta_n^{-2}O\big(n^{-1}\Delta_nU_n^\beta+n^{-1}(\Delta_nU_n^2)^{2m}+n^{-2}\Delta_nU_n^2\big).
\end{align*}
For the last bound we have inserted the sum from \eqref{EqCovExpans} whose integral vanishes by the properties \eqref{Eqwj2} of $w$, the order $U_n=o(\Delta_n^{-1/2})$ and the bound
$\E[\abs{\int(\cdots)^2dw}^2]=O(\max_u\E[(\cdots)^4])$ by Jensen's inequality.
As in the preceding proof, we have $P({\cal G}_n)\to 1$ and thus uniformly under $H_0$
\begin{align*}
T_n&=O(U_n^{-\rho})+O_P\big((n\Delta_n)^{-1/2}U_n^{-\beta_0}(U_n^{\beta_0/2} +\Delta_n^{m-1/2}U_n^{2m}+n^{-1/2}U_n)\big)\\
&=O(U_n^{-\rho})+O_P((n\Delta_n)^{-1/2}(U_n^{-\beta_0/2}+\Delta_n^{m-1/2}U_n^{2(m-\beta_0/2)})).
\end{align*}
Let us suppose $\rho<\beta_0/2$ such that $U_n^{-\beta_0/2}=o(U_n^{-\rho})$ and by balancing the first and last rate we choose the maximal frequency
\begin{equation}\label{EqUBGI2}
U_n=n^{1/(4m+2\rho-2\beta_0)} \Delta_n^{-(m-1)/(2m+\rho-\beta_0)}.
\end{equation}
Exactly as before we thus derive the following result for the same test, but at a different frequency and with moment conditions on the weights.

\begin{proposition}
Suppose $\rho<\beta_0/2$ and that $w_1$, $w_2$ satisfy \eqref{Eqwj2} for some $m\ge 2$. Consider the test on \eqref{EqTheta0BGI} and \eqref{EqTheta1BGI} that rejects if
\[ \abs{T_n}>\frac{K+\eps}{W_2(\beta_0)}(U_n^{-\beta}+U_n^{-\rho}
(\norm{w_1}_{TV}+\norm{w_2}_{TV}Q(\beta_0))).
\]
with $U_n$ from \eqref{EqUBGI2} and any $\eps\in (0,1)$. Then the test
has asymptotic level $0$ and for the separation rate $r_n\to 0$ such that
\[ \frac{r_n}{U_n^{-\rho}}=\frac{r_n}{n^{\rho/(4m+2\rho-2\beta_0)} \Delta_n^{-(m-1)\rho/(2m+\rho-\beta_0)}}\to\infty
\]
the test attains on $\Theta_{1,n}$ the asymptotic power 1.
\end{proposition}

For the high-frequency case  $\Delta_n=\frac1n$ the separation rate is $n^{-\rho(m-1/2)/(2m+\rho-\beta_0)}$, which for $m=2$ ($w_j$ only filter out $u^2$, which we need anyway) yields $n^{-3\rho/(8+2\rho-2\beta_0)}$ while for $m\to\infty$ the separation rate approaches $n^{-\rho/2}$. Note that by the Weierstra{\ss} approximation theorem no weight $w$ can satisfy \eqref{Eqwj2} for all $m$, but for all finite values of $m$ such weights  exist. Given the bound $\rho<\beta_0/2$, we can thus almost attain the rate $n^{-\beta_0/4}$ provided the bias due to the nuisance part only grows with exponent $\rho=\beta_0/2$ in $\Re(\psi)$ and the weights have a high number $m$ of vanishing moments.

This rate consideration is extremely interesting
because even in the simple parametric model where the sum of a Brownian motion and a $\beta_0$-stable process is observed, the optimal rate is
$n^{-\beta_0/4}(\log n)^{\beta_0/4-1}$ (Ait-Sahalia and Jacod \cite{ASJJ2012}). The authors provide a procedure  which for very general semi-martingales estimates a generalized jump index $\beta_0$ with a rate that is significantly slower unless $\beta_0$ is close to zero, compare the discussion in their Section 6.2.
The procedure is based on counting increments that are larger than for the pure Brownian case, their truncation value $u_n$ plays a role as the inverse frequency $U_n^{-1}$ in our approach. In a second step the authors consider also the estimation of secondary and further successive Blumenthal-Getoor-type indices which suffers, however, from even slower rates of convergence.

Another kind of jump intensity measure has been considered by Trabs \cite{Trabs2012} for pure jump self-decomposable processes where the density of the L\'evy measure behaves like $\nu(x)\thicksim \alpha \abs{x}^{-1}$ for $x\to 0$ for some $\alpha\ge 0$, possibly different for $x\downarrow 0$ and $x\uparrow 0$. A well known example is furnished by Gamma processes and the value $\alpha$ determines many probabilistic properties of the self-decomposable processes. Estimating $\alpha$ is possible with polynomial rates, even for option price data, and the rate itself depends on $\alpha$.

In view of the preceding results and discussion let us conclude by asking the imminent question, for which there is no satisfactory understanding yet, leaving aside the even more challenging questions for general semi-martingales.

\begin{problem}
Construct an estimator or a test for $\beta_0$ which achieves exactly the parametric rate for $\beta_0$-stable processes. If this is not possible, provide a lower bound proof for that.
\end{problem}

\section{Nonparametric test on the jump measure}

\subsection{General approach}

Inference on the jump measure as a natural nonparametric object is another key topic of current research. If we are not interested in the realized jumps, but in their intensity, we have to assume an asymptotically infinite time horizon $n\Delta_n\to\infty$. Concerning testing problems there are again a variety of interesting hypothesis formulations. We restrict ourselves to a natural form, based on the observation that small jumps and volatility interfere and from a probabilistic side the finite measure
\[ \nu_\sigma(dx)=x^2\nu(dx)+\sigma^2\delta_0(dx),\]
which combines the downweighted jump measure at zero with a point measure of size $\sigma^2$ at zero, determines the topology in the law of infinitely divisible distributions: if a sequence $(\nu_\sigma^{(m)})_m$ converges weakly to $\nu_\sigma$ and the drifts satisfy $b_m\to b$, then also the finite-dimensional distributions of the corresponding L\'evy processes converge, see e.g. Sato \cite{Sato1999} or Neumann and Rei{\ss} \cite{NeuReiss2009}. The measure $\nu_\sigma$ is also quite tractable as
\[\psi''(u)=-\sigma^2+\int (-x^2)e^{iux}\nu(dx)=-{\cal F}\nu_\sigma(u)\]
holds. We thus aim at testing the null hypothesis $H_0:\nu_\sigma=\nu_{\sigma,0}=x^2\nu_0(dx)+\sigma_0^2\delta_0(dx)$ for some given jump measure $\nu_0$ and volatility $\sigma_0^2$. In classical nonparametrics, e.g. in regression, the local alternative would be formulated by postulating an $L^2$-distance of rate $r_n$ and a given smoothness class, cf. Ingster and Suslina \cite{IngsterSus2003}. Here in the case of finite measures, we shall also consider the case of point measure to encompass the volatility.

Our test is based on the smoothed $L^2$-statistics
\[ T_n=\int \abs{(\hat\psi_n''(u)-\psi_{0}''(u)){\cal F}K_h(u)}^2du,\]
where $K:\R\to\R$ is some kernel function, i.e. $K\in L^2(\R)$ with $\int K=1$, and $K_h(x):=h^{-1}K(x/h)$ for some bandwidth $h>0$. Moreover, we require $K$ to be band-limited, more specifically $\supp({\cal F}K)\subset[-1,1]$, such that the integral only extends over the interval $[-h^{-1},h^{-1}]$. If the true measure is $\nu_{\sigma,1}$ and if the empirical characteristic function is close to the true one, then by Plancherel's theorem $T_n$ measures basically the distance $\int (K_h\ast (\nu_{\sigma,1}-\nu_{\sigma,0}))(x)^2dx$. When the bandwidth $h$ tends to zero, the distance tends to infinity unless $\nu_{\sigma,1}-\nu_{\sigma_0}$ possesses an $L^2$-density whose squared $L^2$-norm is then obtained. As a side remark we can also view ${\cal F}K_h(u)={\cal F}K(hu)$ as a weight function $w(U_n^{-1}u)$ with $U_n=h^{-1}$ and interpret the test statistics just as a frequency weighted average in the spirit of the preceding section.

To analyse the properties of $T_n$ under $H_0$, we linearize the dependence of $\psi''$ on $\phi$. We set
\[g:=\frac{\hat\phi_n-\phi_n}{\phi_n}\Rightarrow g''=\frac{(\hat\phi_n-\phi_n)'' -2 (\hat\phi_n-\phi_n)'\Delta_n\psi' - (\hat\phi_n-\phi_n)(\Delta_n\psi''-\Delta_n^2(\psi')^2)}{\phi_n}
\]
and obtain as $\delta_n:=\max_{j=0,1,2}\abs{g^{(j)}(u)}\to 0$
\[ \hat\psi_n''(u)-\psi''(u)=\frac{1}{\Delta_n}\frac{g''(u)(1+g(u))-g'(u)^2}{(1+g(u))^2}
=\Delta_n^{-1}\big(g''(u)+O(\delta_n^2)\big).
\]
%Evaluating the dominant term, we arrive at
%\[ \hat\psi_n''(u)-\psi_n''(u)=\frac{1}{\Delta_n}\Big(\frac{(\hat\phi_n-\phi_n)''(u)}{\phi_n(u)} + \frac{(\hat\phi_n-\phi_n)'(u)}{\phi_n(u)} \Delta_n\psi_n'(u) -2 \frac{(\hat\phi_n-\phi_n)(u)}{\phi_n(u)} \Delta_n^2\psi_n'(u)^2 +O(\max_{j=0,1,2}g^{(j)}(u)^2)\big).
%\]
Assuming finite  moments of order 8, i.e. $\int x^8\nu(dx)<\infty$, and using \[\Var_{\C}(\hat\phi_n^{(j)}(u))\lesssim n^{-1}(\Delta_n+{\bf 1}(j=0)),\, \E[\abs{\hat\phi_n^{(j)}(u)-\phi_n^{(j)}(u)}^4]\lesssim n^{-2}(\Delta_n+{\bf 1}(j=0))
\]
for derivatives of order $j=0,1,2$, we infer
\[ \abs{\hat\psi_n''(u)-\psi''(u)}^2=O_P\Big( \frac{1}{n\Delta_n^2\abs{\phi_n(u)}^2} \big(\Delta_n + \Delta_n^3\psi'(u)^2 +  \Delta_n^2\psi''(u)^2 +\Delta_n^4\psi'(u)^4+n^{-1}\abs{\phi_n(u)}^{-2}\big)\Big)
\]
uniformly in $u$. Together with $\norm{{\cal F}K_h}_{L^2}^2=O(h^{-1})$ and $\abs{\psi'(u)}\lesssim\abs{u}$, $\abs{\psi''(u)}\lesssim 1$ for $\abs{u}\to\infty$ this implies uniformly over $H_0$ for $h\in(0,1)$
\begin{equation}\label{EqTnH0}  T_n=O_P\Big(\max_{\abs{u}\le h^{-1}} \frac{h^{-1}}{n\Delta_n\abs{\phi_{n,0}(u)}^2} \big(1 + \Delta_n^2h^{-2}+\Delta_n^3h^{-4} +(n\Delta_n)^{-1}\abs{\phi_{n,0}(u)}^{-2}\big)\Big).
\end{equation}
%On the other hand, under a true measure $\nu_{\sigma,1}$ we have the property
%\[ T_n\ge \norm{K_h\ast (\nu_{\sigma,1}-\nu_{\sigma_0}))}_{L^2}^2- \norm{K_h\ast (\nu_{\sigma,1}-\nu_{\sigma_0}))}_{L^2}O_P\Big(\max_{\abs{u}\le h^{-1}} \frac{h^{-1}}{n\Delta_n\abs{\phi_{n,1}(u)}^2} \big(1 + \Delta_nh^{-2} +(n\Delta_n)^{-1}\abs{\phi_{n,1}(u)}^{-2}\big)\Big).
%\]
The choice of the bandwidth $h$ is at our disposal and we shall discuss the  low-frequency setting and the case $\Delta_n\to 0$ separately. To remain concise, formal proofs are omitted. They follow along the lines of the preceding sections.

\subsection{Asymptotically small observation distance}

If $\Delta_n\to 0$ we may restrict to $h_n\ge\Delta_n^{1/2}$ which yields $\max_{\abs{u}\le h_n^{-1}}\abs{\phi_{n,0}(u)}^{-2}=O(1)$. The advantage is that the bound \eqref{EqTnH0} under $H_0$  simplifies considerably:
\[ T_n=O_P\Big(\frac{h_n^{-1}}{n\Delta_n} \Big).
\]

We can gain efficiency when $\norm{\nu_{\sigma,0}-\nu_{\sigma,1}}_{H^s}\le R$, that is $\nu_{\sigma,0}-\nu_{\sigma,1}$ has a Lebesgue density which lies in an $L^2$-Sobolev ball of regularity $s>0$ and radius $R>0$. Then standard approximation theory gives
\begin{align*}
\norm{K_h\ast(\nu_{\sigma,0}-\nu_{\sigma,1})}_{L^2} &\ge \norm{\nu_{\sigma,0}-\nu_{\sigma,1}}_{L^2}-\norm{\nu_{\sigma,0}-\nu_{\sigma,1}-K_h
\ast(\nu_{\sigma,0}-\nu_{\sigma,1})}_{L^2}\\
&= \norm{\nu_{\sigma,0}-\nu_{\sigma,1}}_{L^2}-O(Rh^s).
\end{align*}
Hence, we obtain for $h_n=(n\Delta_n)^{-1/(2s+1)}$ (assuming $h_n\gtrsim\Delta_n^{1/2}$, i.e. $\Delta_n\lesssim n^{-2s/(2s+3)}$) and for $\norm{\nu_{\sigma,0,n}-\nu_{\sigma,1,n}}_{L^2}\ge r_n$ with $r_n(n\Delta_n)^{s/(2s+1)}\to\infty$
\[  \sqrt{n\Delta_nh_n}\norm{K_{h_n}\ast(\nu_{\sigma,0,n}-\nu_{\sigma,1,n})}_{L^2}\ge \sqrt{n\Delta_nh_n}(r_n-O(Rh_n^s))\to\infty.
\]
Under $H_0$ the statistics $h_nn\Delta_nT_n$ is bounded in probability and any test that rejects if $n\Delta_nh_nT_n>\kappa_\alpha$ for suitable $\kappa_\alpha>0$ has asymptotically uniform level $\alpha\in(0,1)$. Under $H_1$ the above bias bound yields $n\Delta_nh_nT_n\to\infty$ in probability such that asymptotically the test has power 1.
For uniform power over $H_1$ we require $\int (x^2\vee x^8)\nu(dx)\le R$ over all triplets in $H_1$ to make the above linearisation and stochastic bound work uniformly (note the decay bound $\Re(\psi(u))\ge -\nu_\sigma(\R)u^2/2\ge -Ru^2/2$).

The rate  $(n\Delta_n)^{-s/(2s+1)}$ is certainly the optimal minimax rate for estimation since the jumps away from zero follow a compound Poisson process and even continuous-time observation on $[0,n\Delta_n]$ of a compound Poisson process with jump density in $H^s(\R)$ does not allow for faster estimation rates (it corresponds to density estimation with $n\Delta_n$ observations). Note that for testing with respect to separating hypotheses in $L^2$ the rates for idealised regression models can be improved to $(n\Delta_n)^{-s/(2s+1/2)}$ (Ingster and Suslina \cite{IngsterSus2003}), which in our setting could only be established if we can show $\Var(T_n)\le h(\E[T_n])^2$ and if we can perform an asymptotic bias correction for $\E[T_n]\not=0$.

A comprehensive theory for estimating the jump density under asymptotically small observation distances has been achieved by Figueroa-Lopez \cite{Fig2009}, using sieve or projection methods on intervals bounded away from zero, excluding the small jumps. For continuous-time observations the jumps on $[0,n\Delta_n]$  are observed directly and for sufficiently small $\Delta_n$ the increments can be used as a proxy for the jump sizes away from zero. The rate $(n\Delta_n)^{-s/(2s+1)}$ is achieved, an adaptive method via model selection is proposed and a minimax lower bound is proved in detail. The crucial condition on the speed $\Delta_n\to 0$ in this paper is more than exponential in $n$ and a much finer probabilistic analysis was needed to relax this condition to  $\Delta_n=o(n^{-1/3})$ (Figueroa-Lopez and Houdr\'e \cite{FigHoudre2009}), to be compared with  $\Delta_n= O(n^{-2/(2s+3)})$ in our case. This study of the discretisation error reveals a marvelous interplay between statistics and probability theory. Inference based on the second derivative $\hat\psi_n''$ of the empirical characteristic exponent has been exploited by Comte and Genon-Catalot \cite{ComteGC2011} to develop also a sieve or projection method with rate $(n\Delta_n)^{-s/(2s+3)}$ for estimating the smooth density $x^2\nu(dx)$ on $\R$ in the case $\sigma=0$. Adaptivity is again achieved by model selection, using sample splitting in addition. Closer related to the testing results obtained here are the confidence statements by Figueroa-Lopez \cite{Fig2011}, but they are again based on approximating continuous-time observation statistics.

Let us point out that for the (weighted) L\'evy measures $\nu_{\sigma,j}$ separation in other topologies than $L^2(\R)$ for their Lebesgue densities are natural. In particular,
we can be interested in the separation rate for different volatilities $\sigma_0^2$ and $\sigma_1^2$ or for the more general point measure case $r_n:=(\nu_{\sigma,0}-\nu_{\sigma,1})(\{x_0\})>0$ for some $x_0\in\R$. If we assume one or finitely many point measures in $\nu_{\sigma,0}-\nu_{\sigma,1}$ and otherwise a smooth density with $H^s$-norm uniformly bounded by $R>0$, we obtain as $h\to 0$
\[ \norm{K_h\ast(\nu_{\sigma,0}-\nu_{\sigma,1})}_{L^2}\ge r_n h^{-1/2}-O(Rh^s).\]
In order to separate $H_1$ from $H_0$ we need $r_n/(n\Delta_n)^{-1/2}\to\infty$ for the same specification as before, that is $h_n=(n\Delta_n)^{-1/(2s+1)}$ and $\Delta_n=O(n^{-2/(2s+3)})$ (here even non-positive smoothness $s\in(-1/2,0]$ is possible).
Instead of point measures we can also consider jump densities with singularities (poles or cusps) of the type $\abs{\nu_{\sigma_0}-\nu_{\sigma,1}}(x)\thicksim r_n\abs{x-x_0}^{1-\beta}$ for $x\to x_0$ and some $\beta<2$, e.g. stable-like behaviour with index $\beta$ of $\nu$ at $x_0=0$. Then $\norm{K_h\ast(\nu_{\sigma,0}-\nu_{\sigma,1})}_{L^2}\ge r_n h^{(3-2\beta)/2}-O(Rh^s)$ holds, assuming $H^s$-norm at most $R$ away from $x_0$ and $s>(3-\beta)/2$. The above analysis yields similarly for $r_n$ the rate $(n\Delta_n)^{-(s+\beta-3/2)/(2s+1)}$.
The rates are better than for differences with smooth densities and for $x_0\not=0$ also optimal for estimation. Yet, they are worse than those obtained by the specifically designed tests from the previous two sections on the volatility and on the jump index  because they are given in terms of the observation time $n\Delta_n$ instead of the observation number $n$. Maybe a tighter bound on the stochastic error is possible for $x_0=0$.

\subsection{Low-frequency observations}

For low-frequency observations $\Delta_n=1$ we cannot follow the previous route. We must choose $h_n\to 0$ and thus accept $\phi_n(h_n^{-1})\to 0$ exponentially fast in $h_n^{-1}$, assuming $\sigma^2>0$. Let us therefore choose $h_n=(\log(n^c))^{-1/2}$ for some small $c>0$ such that \eqref{EqTnH0} under $H_0$ becomes
\[T_n=O_P\big(h_n^{-1}n^{c\sigma_0^2-1}(h_n^{-4} +n^{c\sigma_0^2-1})\big)=O_P(n^{-\eps})\text{ for }c<(1-\eps)/\sigma_0^2.
\]
In the smooth jump density case $\norm{\nu_{\sigma,0}-\nu_{\sigma,1}}_{H^s}\le R$ it is essential that $r_n-O(Rh_n^s)\ge 0$ and we necessarily obtain a logarithmic separation rate $(\log n)^{-s/2}$. For point measure alternatives such that $\norm{K_h\ast(\nu_{\sigma,0}-\nu_{\sigma,1})}_{L^2}\ge r_n h^{-1/2}-O(Rh^s)$, imposing a uniform bound $\bar\sigma^2$ on $\sigma^2$, we obtain
with $c\bar\sigma^{2}<1$ the separation rate $(\log n)^{-(s+1/2)/2}$.

In all these cases the stochastic error is negligible compared to the bias if $c>0$ is chosen sufficiently small. This phenomenon is well known for so-called severely ill-posed statistical inverse problems like deconvolution with a Gaussian density. In fact, $\Delta\hat\psi_n'(u)=\hat\phi_n'(u)/ \hat\phi_n(u)$ implies in the bounded variation case ($\int_{-1}^1\abs{x}\nu(dx)<\infty$ and $\sigma=0$) for the law $P_\Delta$ of $X_\Delta$
\[ xP_\Delta(dx)=\Delta\Big(\Big(b-\int x\nu(dx)\Big)\delta_0(dx)+x\nu(dx)\Big)\ast P_\Delta(dx).\]
Since we have empirical access to $P_\Delta$, the characteristics are thus identified by deconvolving $xP_\Delta(dx)$ with $P_\Delta(dx)$. In classical nonparametric statistics this inversion is usually also performed by division in the Fourier domain. The main new features from a pure statistical point of view are that $xP_\Delta(dx)$ and $P_\Delta(dx)$ depend on the same unknown law (the problem is also coined {\it auto-deconvolution}), the empirical versions are highly correlated and the measure of interest $x\nu(dx)$ has some L\'evy-specific structure. A similar, but slightly more involved convolution equation is obtained for general L\'evy processes using $\hat\psi_n''(u)$. The plausible intuition
for this connection to deconvolution is that at low frequency sums of (an unknown number of) jumps are observed in the increments, moreover diluted by a Gaussian component, and it is at first sight very remarkable that consistent inference is possible at all.

The problem of estimating the jump distribution from observations of a compound Poisson process has come up early in insurance mathematics and is called decompounding. Inference on  the distribution function of $\nu$ has been first studied rigorously by Buchmann and Gr\"ubel \cite{BuchGr2003} who invert the nonlinear equation $P_\Delta=\exp^\ast(\Delta\nu)e^{-\Delta\lambda}$ in terms of the convolution exponential $\exp^\ast$. A first natural extension is to consider L\'evy processes of finite intensity, hence superpositions of compound Poisson processes and Brownian motion. Here, the main features of the general L\'evy case already appear and the weight function approach for the empirical characteristic exponent, initiated in Belomestny and Rei{\ss} \cite{BelReiss2006} for option price data, was successfully applied to low frequency data by Gugushvili \cite{Gug2009}. In both cases, logarithmic rates are obtained in the case $\sigma>0$ for all parameters in the characteristic and for the intensity $\lambda=\nu(\R)$. The optimal convergence rates become slower along the order $\sigma^2,b,\lambda,\nu$ which is clear from the influence on the characteristic exponent for high frequencies.

The problem of inference on $\nu_\sigma$ for general L\'evy processes was addressed by Neumann and Rei{\ss} \cite{NeuReiss2009} and extended by Kappus and Rei{\ss} \cite{KapReiss2010}, using a minimum-distance estimator for empirical characteristic functions. The problem has been studied more thoroughly and with a focus on adaptive estimation by Kappus \cite{Kappus2012}, improving results by Comte and Genon-Catalot \cite{ComteGC2010}, which brings in new perspectives and methodology on model selection with unknown heteroskedastic noise structures. Nickl and Rei{\ss} \cite{NicklReiss2012} then address the fundamental inference problem on the distribution-type function $N$ of $\nu$, defined by $N(x)=\nu((-\infty,x])$ for $x<0$ and $N(x)=\nu((x,\infty))$ for $x>0$. Given a slow polynomial decay of the characteristic function, a uniform central limit theorem for a natural plug-in estimator $\hat N_n(x)$ is provided, i.e. for the corresponding empirical process $(\sqrt{n}(\hat N_n(x)-N(x)),\abs{x}>\zeta)$, $\zeta>0$, very much  in the spirit of the classical Donsker theorem for the empirical distribution function. It is surprising that this result demands a very fine analysis of the underlying deconvolution operator by means of harmonic analysis and Fourier-multiplier theory as well as the theory of smoothed empirical processes.

Inference on the L\'evy measure has many parallels with inference on the distribution of an i.i.d. sample and already very natural questions inspire and stimulate research in other areas like the theory of jump processes and semi-martingales, adaptation techniques, empirical processes or Fourier-integral operators. This calls for further exploration, both from a theoretical and applied perspective.

\begin{problem}
Solve the basic statistical inference questions also for the L\'evy measure like natural analogues of goodness-of-fit tests (Kolmogorov-Smirnov, Cram\'er-von Mises), semiparametric efficiency, empirical process classes. Focus especially on the small jumps.
\end{problem}

\end{document}